\documentclass[12pt]{amsart}

\usepackage{amssymb}
\usepackage{amsmath}
\usepackage{array}

\newcommand{\bS}{\mbox{$\mathbb S$}}

\newcommand{\Pf}{{\em Proof}. }
\newcommand{\EPf}{\hfill$\square$}
\newcommand{\Z}{\mbox{$\mathbb Z$}}
\newcommand{\R}{\mbox{$\mathbb R$}}
\newcommand{\C}{\mbox{$\mathbb C$}}
\newcommand{\Q}{\mbox{$\mathbb H$}}

\newcommand{\SU}[1]{\mbox{$\mathbf{SU}(#1)$}}
\newcommand{\U}[1]{\mbox{$\mathbf{U}(#1)$}}
\newcommand{\SP}[1]{\mbox{$\mathbf{Sp}(#1)$}}
\newcommand{\SO}[1]{\mbox{$\mathbf{SO}(#1)$}}
\newcommand{\OG}[1]{\mbox{$\mathbf{O}(#1)$}}
\newcommand{\Spin}[1]{\mbox{$\mathbf{Spin}(#1)$}}

\newcommand{\G}{\mbox{$\mathbf{G}_2$}}

\newcommand{\tref}[1]{Theorem~\ref{#1}}
\newcommand{\cref}[1]{Corollary~\ref{#1}}
\newcommand{\pref}[1]{Proposition~\ref{#1}}

\newcommand{\lref}[1]{Lemma~\ref{#1}}

\newtheorem{thm}{Theorem}[section]
\newtheorem*{thm*}{Theorem}
\newtheorem*{thmmain*}{MAIN THEOREM}
\newtheorem{lem}[thm]{Lemma}
\newtheorem{cor}[thm]{Corollary}
\newtheorem{prop}[thm]{Proposition}
\newtheorem*{prop*}{Proposition}

\theoremstyle{definition}

\theoremstyle{remark}
\newtheorem{rem}{Remark}[section]
\newtheorem{ex}[rem]{Example}
\newtheorem{quest}[rem]{Question}

\title{Isometric actions on spheres\\ with an orbifold quotient}

\author{Claudio Gorodski and Alexander Lytchak}

\address{Instituto de Matem\'atica e Estat\'\i stica, Universidade de 
S\~ao Paulo, Rua do Mat\~ao, 1010, S\~ao Paulo, SP 05508-090, Brazil}
\email{gorodski@ime.usp.br}

\address{Mathematisches Institut, Universit\"at zu K\"oln, 
Weyertal 86-90, 50931 K\"oln, Germany}
\email{alytchak@math.uni-koeln.de}

\thanks{The first author was partially supported by the
CNPq grant 303038/2013-6 and the FAPESP project 2011/21362-2.}

\date{\today}

\begin{document}

\begin{abstract} 
We classify representations of  compact connected Lie
groups whose induced action on the unit sphere 
has an orbit space isometric to a Riemannian orbifold.
\end{abstract}

\maketitle

\section{Introduction}

In this paper we classify all orthogonal representations of compact 
connected   
groups $G$ on Euclidean unit spheres
$\mathbb S^n$  for which the quotient $\bS^n/G$ is a Riemannian orbifold.
We are going to call such representations \emph{infinitesimally polar}, since they 
can be equivalently defined  by
the condition that slice representations at all non-zero vectors 
are \emph{polar}.   The interest in such representations is two-fold.

On the one hand, one can hope to isolate an interesting class of representations in this way.  This class contains all \emph{polar representations}, which can be defined by the property that the corresponding quotient space  $\mathbb S ^n /G$ is an orbifold of constant curvature $1$.  Polar representations very often play an important and special role in  geometric questions concerning representations (cf.~\cite{PT2,BCO}), and the class  investigated in this paper consists of  closest relatives of polar representations.

On the other hand, any such quotient has positive curvature and all geodesics in
the quotient are closed.  Any of these two
properties is extremely interesting and in both classes the number of known manifold examples is very limited (cf.~\cite{Ziller,Be}).
 We have hoped to obtain new examples of positively curved manifolds and  of manifolds with closed geodesics as universal orbi-covering of such spaces.
Should the orbifolds be \emph{bad} one could still  hope to obtain new examples of interesting orbifolds.  For instance, the two-dimensional
weighted complex projective spaces  are very special  examples of rotationally invariant singular Zoll surfaces (\cite[chapter~4]{Be} 
and~\cite{GUW}).  And the four-dimensional
weighted quaternionic projective space $\mathbb S^7 /\SU2$, where 
$\SU2$ acts by the irreducible representation of complex
dimension~$4$, is the so called Hitchin orbifold $O_3$ 
(cf.~\cite{ziller07} and the references therein). 
This orbifold with a different positively curved metric is the 
starting point of the only  new positively curved manifold discovered 
in the last decades~\cite{Dear,GVZ}.  We have hoped that new interesting 
examples of orbifolds may arise in this way.
 
Unfortunately, this hope was not fulfilled. Our main result shows that no interesting new examples  occur.  Namely, we prove:

\begin{thm}\label{firstthm}
Let $\rho$ be an orthogonal representation of a compact Lie group $G$ on a Euclidean space $\R ^{n+1}$. Assume that
the quotient $X=\mathbb S^n /G$ is a Riemannian orbifold.   Then the universal orbi-covering $\tilde X$ of 
$X$ is either a weighted complex or a weighted  quaternionic projective space, or $\tilde X$ has constant curvature $1$ or  $4$. 
\end{thm}

As mentioned above, $\tilde X$ has constant curvature $1$ if and only if the action of $G$ is polar, a very well   understood situation~\cite{D}.
It turns out that $\tilde X$ is a weighted projective space if and only if it has dimension $2$ or $G$ acts almost freely and has rank $1$, again
two very well known cases~\cite{str,gromollsph}.   

Almost free actions 
and polar representations provide orbifolds of arbitrary dimension.
On the other hand, the remaining case  of quotients of constant curvature $4$ can  occur only in low cohomogeneity, maybe the most surprising consequence of our classification.

\begin{thm} \label{secondthm}
Let $X=\mathbb S^n /G$ be a Riemannian orbifold.  If $G$ has rank at least $2$ and the representation is not polar, then
the dimension $k$ of $X$ satisfies $2\leq k \leq 5$.  
\end{thm}

 This result becomes more  surprising if  one compares it with  the world 
of non-homogeneous singular Riemannian  foliations on spheres.  
Recent examples of M. Radeschi~\cite{Radeschi2} 
show that there are singular Riemannian foliations
on round spheres $\mathbb S^n$, for $n$ large enough, 
such that the quotient space is isometric 
to a  hemisphere of curvature $4$, which can be of arbitrary large 
dimension!  

On the other hand, for more standard
examples of infinitesimally polar actions 
the differences are not that big. For polar foliations, 
it is known 
that the only quotients 
that arise, arise as quotients of group representations~\cite{Th5}.
Moreover, in higher codimensions
essentially all  polar  foliations are homogeneous~\cite{Th2}.    The case of  $G$ having rank one and acting almost freely  corresponds to the case of regular  Riemannian 
foliations. All such foliations  but the $7$-dimensional Hopf fibration 
are homogeneous (\cite{LW} and the literature therein).  
We refer also to  \cite{Radeschi} for a related result.

As it turns out, there are only very few representations satisfying 
the assumptions of \tref{secondthm} for $k\geq 3$, the case $k=2$  having been previously classified in \cite{str}.
 At least if we restrict the 
representation to the connected component,
all such representations are of a similar type: they are,
up to a circle factor,
the double of a standard vector representation. 
In the Table~1 and further in the text we will use 
the following notation:
by  $\mathbb S^k (r)$,  $\mathbb S^k _+ (r)$, 
$\mathbb S^k  _{++}  (r)$, $\mathbb S^k  _{+++} (r)$ we denote the round sphere of constant curvature $\frac {1} {r^2}$ quotiented by the group $\Gamma$ which is 
respectively  generated
by $0$, $1$, $2$, $3$ commuting reflections. We call the corresponding spaces the
sphere, hemisphere, quarter-sphere and eighth-sphere of curvature $\frac 1 {r^2}$, respectively.

\begin{thm} \label{thirdthm}
Let $\rho: G \to \SO {n+1}$ be an orthogonal representation of a 
compact connected Lie group $G$
such that the quotient space $X=\mathbb S ^n /G$ is a Riemannian orbifold of dimension $k$.   Assume that $\rho$ is non-polar, $G$ has rank at least $2$ and $k\geq 3$.
Then $\rho$ is one of the following:
 \setlength{\extrarowheight}{0.1cm}
\[ \begin{array}{|c|c|c|c|}
\hline
G & \rho & \textsl{Conditions} & \textsl{Orbit space}\\
\hline
\Spin9 & \R^{16}\oplus\R^{16} & - & \bS^3_{++}(\frac12)\\
\SU n & \C^n\oplus\C^n & n\geq3&  \bS^3_+(\frac12)\\
\U n & \C^n\oplus\C^n & n\geq2& \bS^3_+(\frac12)\\
\SP n & \Q^n\oplus\Q^n & n\geq2 & \bS^5_+(\frac12)\\
\SP n\U1 & \C^{2n}\oplus\C^{2n} & n\geq2 & \bS^4_{++}(\frac12)\\
\SP n\SP1 & \R^{4n}\oplus\R^{4n} & n\geq2 & \bS^3_{++}(\frac12)\\
\mathbf T^2\times \SP n & \C^{2n} \oplus \C^{2n} & n\geq2& \bS^3_{+++}(\frac12)\\
\hline
\end{array} \]
\begin{center}
\textsc{Table 1: Quotients of constant curvature~$4$}
\end{center}
\end{thm}

We can now summarize our classification:
\begin{thm} \label{thmsum}
Let a compact Lie group $G$ act effectively and isometrically on the  unit
sphere 
$\mathbb S^n$.  Let $\rho: G^0 \to\SO {n+1}$ be the corresponding 
representation of the identity component.
The quotient $\mathbb S^n /G$ is a Riemannian orbifold if and only if one of 
the following cases occur for $\rho$:
\begin{itemize}
\item The representation $\rho$ is polar;
\item $G^0 =S^1$ and $\rho$  is a sum of non-trivial 
irreducible representations;
\item $G^0 =\SU2$ and  $\rho$ is a sum of non-trivial  irreducible 
quaternionic representations;
\item $\rho$ has cohomogeneity $3$;
\item $\rho$ is described in \tref{thirdthm}.
\end{itemize}
\end{thm}

All good quotient orbifolds appearing in \tref{thmsum} have curvatures 
between~$1$ and~$4$
and all bad quotient orbifolds appearing in \tref{thmsum} have diameter~$\pi /2$, 
hence curvature~$\leq 4$ at some points.    
We are not aware of a geometric  reason, why the curvature cannot be constant 
equal to $9$, or just larger than $4$ at every tangent plane.
Due to \cite{LT}, for  any representation $\rho:G \to \SO {n+1}$ the quotient  
$X=\mathbb S^n /G$ always has regular points 
with arbitrary large curvatures at some plane, 
unless $X$ is a Riemannian orbifold, hence
described above.  On the other hand, it seems to be an interesting question, 
how large the infimum of the curvatures can be in a general  
quotient space $X=\mathbb S^n /G$.  
We hope to discuss this question closely related to \cite{Gowan,Greenwald} in a forthcoming paper.


Many parts of our proof apply to non-homogeneous  singular Riemannian foliations. Some observations  may be of independent
interest, for instance the  following:

\begin{thm} \label{good}  
Let $X$ be a compact positively curved Riemannian orbifold of 
dimension at least $3$. If $X$ has non-empty boundary as an Alexandrov space, then $X$  is a good orbifold.  Moreover,  the universal orbi-covering  
of $X$ is diffeomorphic to a sphere.
\end{thm}

In dimension $2$ the above theorem is not true:
a two-dimensional disc with a singularity at the boundary 
(half of a tear-drop orbifold)
is a bad  orbifold  and can be equipped with a 
positively curved metric. Indeed, this
orbifold arises as the orbit space of
certain isometric group actions on unit spheres (compare Example \ref{example}).
 In view of this result, \cite[Theorem~2.10]{fgt} and~\cite{fg},
it seems reasonable to ask the following:

\begin{quest}
Do all orbifolds  from \tref{good} carry a metric of constant curvature?
\end{quest}
  
Another observation 
that may be of independent interest in other contexts  
is a reduction procedure similar to the reduction of the principal isotropy 
group \cite{lr,str,grove-searle,GL}.  Namely, we show in 
Section~\ref{secorbiquot} that one can view any stratum~$S$ in a given quotient 
space $M/G$ as an open subset of a (usually non-injective)
isometric immersion of 
another quotient space $F/N$, where $F$ is a closed and totally geodesic 
submanifold of $M$ and $N$
a subquotient group of $G$. In the particular case in which $S$ is a closed stratum, 
this amounts to the easy statement that the whole manifold $S$ is isometric to some 
quotient $F/N$ as above. In our opinion, it is exactly this point,
that does not have a similar for general singular Riemannian  foliations,
which is responsible for the  great variety of examples  in~\cite{Radeschi2}.  
For instance 
(cf.~\cite[Subsection 6.2]{Radeschi2}), 
assume for the moment that our quotient $\mathbb S^n /G$ is a hemi-sphere of 
constant curvature greater than~$1$. 
Then the boundary sphere $S$ is a closed stratum of our quotient and thus,
by the observation above,  the  manifold $S$ is itself isometric to the quotient of 
another sphere,  $S=\mathbb S^l /N$.   
But then $N$ must have rank one and $S$ must be $\C P ^1$ or $\mathbb H P ^1$ hence 
of dimension $2$ or $4$.  
If a similar reduction procedure existed for general singular Riemannian foliations, 
we would get a fibration $\mathbb S^l \to S$, and the dimension of $S$ would be $2$,
$4$ or~$8$. However, according to~\cite{Radeschi2}, 
the dimension of $S$ can be an arbitrary number.

This work has the following structure.   
In Section \ref{secprel}  we collect basics about 
general quotient spaces and Riemannian orbifolds.   In the last part of the section we explain that
the usual reduction of the principal isotropy group  may be used to reduce the 
classification of possible quotient orbifolds $\mathbb S^n/G$ to the case in which all isotropy groups of $G$ are products of groups of rank one, with very special slice representations (\lref{slice}).  It would have been possible to use this information together with 
some algebraic tools (cf.~\cite{S,Howe})
to prove \tref{thmsum} essentially 
by means of representation theory only, however in a rather elaborate way.
We  choose a different,  more geometric path. In Section  \ref{secdiscuss} we discuss
the geometry of all  the  quotient spaces appearing in \tref{thmsum}, in particular,   verifying the last column of  Table~1.    In Section \ref{secorbifold} we first discuss some general facts
about Coxeter orbifolds, and describe in \lref{goodness} under which conditions they are good orbifolds.  Then we use  convexity arguments   from the Soul Theorem and the theorem of Fraenkel-Petrunin to prove \tref{good}.  Already at this point we know that topologically
no interesting new manifolds can occur in our classification.  
In Section \ref{secorbiquot}, we use the reduction procedure described above
to  begin 
with the investigation of the structure of possible quotient spaces $\mathbb S^n /G$ isometric to a $k$-dimensional Riemannian orbifold $X$, with $k\geq 3$.  Under the assumption that $X$ has non-empty  boundary, i.e., that $X$ is not a weighted projective space, we apply the reduction
 and find another spherical quotient $\mathbb S^m/N$, which is a Riemannian orbifold $Y$ of dimension $(k-1)$.  Moreover, the action of $N$ is non-polar if the action of $G$ is non-polar.   In the last section we use this as the main inductive step in our argument.  Before embarking on the proof of  the main result, we  discuss in Section  \ref{special-red} a special case. In \pref{red} we discuss which sums of two  representations of cohomogeneity $1$ satisfy the assumption of \tref{thmsum}, obtaining exactly the representations listed in Table 1. 
In Section~\ref{seclast} 
we  prove by induction on the dimension $k$ of the quotient $\mathbb S^n/G$ a
 central result, which is the last piece of information that we need 
to finally deduce in Section~\ref{secverylast}
the remaining claims. 
Note, that by the general strata-reduction argument from Section \ref {secorbiquot} it suffices to  prove the results for $k\leq 6$. In principle, one could apply brutal force and write down all representations  of cohomogeneity $\leq 7$. Our proof is more geometric:
we use the inductive structure  to show that the representation can be reducible only as a sum of two representations of cohomogeneity $1$, which finishes the reducible case, due to the Section~\ref{special-red}.   Finally, we use again the inductive structure to reduce the number of possible irreducible  candidates to very few ones, which are then excluded case-by-case.

\emph{For the sake of convenience,    we generally assume that our 
representations are almost effective.}

\medskip

\noindent{\emph{Acknowledgment.}} The second named author would like to thank Gudlaugur Thorbergsson for very constructive discussions at the time
of the work on~\cite{LT}, during which 
the idea for the present paper was born.


\section{Preliminaries}  \label{secprel}
\subsection{Strata of isometric group actions}\label{Stratasubsec}
Let $G$ be a closed group of isometries of a complete connected 
Riemannian manifold 
$M$.   Let $X$ denote the quotient metric space $M/G$.  
The manifold $M$ is stratified 
by $G$-invariant submanifolds consisting of points whose isotropy 
groups are conjugate to each other.  The projection of any such stratum 
to~$X$ is called an \emph{isotropy stratum} of~$X$.    
Any isotropy stratum is an (often non-complete and disconnected)
 Riemannian manifold, which is a locally convex subset
of the Alexandrov space $X$. Any closure of any isotropy stratum
 is an \emph{extremal subset} of the 
Alexandrov space $X$ (cf. \cite[section~4.1]{Petruninsemi}).
From the geometric point of view,  the more natural stratification 
of~$X$ consists of the \emph{connected components}
of isotropy strata defined above.
These components (which also coincide with the components of \emph{normal isotropy strata}~\cite[Introduction]{S})  can also  be defined as the connected components of the set of points  in $X$ with isometric tangent  cones.
 Hence, they are determined uniquely by the metric space  $X$. In particular, they  do not depend on the  presentation of the metric space $X$ as a quotient $X=M/G$. 
We will call them the \emph{metric strata} or, simply, the \emph{strata}
of $X$.  Referring to the general structure of Alexandrov spaces, these closures are exactly the \emph{primitive} extremal subsets of $X$, and thus correspond to  the canonical metric stratification of the Alexandrov space $X$~\cite[Section~4.1(8)]{Petruninsemi}.

There is a unique full-dimensional isotropy stratum of $X$, 
corresponding to the minimal isotropy group, called 
the \emph{principal isotropy group}.
This  stratum  is always connected, convex, open and dense in $X$. It is called the \emph{principal stratum} and consists
of all \emph{principal orbits} of the $G$-action. 

We would like to mention, that the structure described above  answers Question~4.6(2) in~\cite{AKLM2}.





\subsection{Polar representations}
We refer to \cite{PT2,Th5} for accounts on polar representations.  
Recall that
a representation  $\rho:G \to \mathbf O(n+1)$  is polar if and only the 
restriction of the representation to the identity component is 
orbit-equivalent to  
the isotropy representation of a symmetric space. This happens if and only 
if the quotient space $\R ^{n+1}/G$ is a flat Riemannian orbifold.
An equivalent  property  is  that the quotient 
$\mathbb S^n /G$ of the corresponding unit sphere  is a Riemannian 
orbifold of constant curvature $1$.  

\subsection{Riemannian orbifolds}  

We assume   
some experience with orbifolds.  
We refer the reader 
to~\cite{Davis} and~\cite{KleinerLott} for the basic background.  
Some information can also be found in \cite[section~3]{GL} and much more information in  \cite{BrH}, pp.~584--618.

A Riemannian  orbifold comes along with a canonical stratification. Each stratum is 
a connected  Riemannian manifold, which is locally convex
with respect to the ambient metric. The closure of any stratum is a union 
of strata. Any Riemannian orbifold  $X$ can be 
written as a quotient of a Riemannian manifold (the orthonormal
frame bundle of~$X$)  under an almost free isometric action of a compact 
Lie group. The canonical stratification of $X$  is then exactly the 
stratification described in Subsection~\ref{Stratasubsec}.
The \emph{boundary} of the orbifold (in the sense of Alexandrov geometry)
is the closure of the  union of strata of codimension $1$.   

Any Riemannian orbifold $X$ has a unique universal orbi-covering $\tilde X$ with a 
discrete isometric action of $\Gamma = \pi_1 ^{orb} (X)$ on $\tilde X$. 
The orbifold $X$ has a non-empty boundary 
if and only if $\Gamma$ contains a \emph{reflection}, i.e., an involution with 
the set of fixed points of codimension~$1$ in~$\tilde X$.

\subsection{Orbifolds as quotients}
We recall from \cite{LT} that for an isometric action of a compact Lie group 
$G$ on a Riemannian manifold $M$, the quotient $X=M/G$ 
is a Riemannian orbifold if and only if all slice representations of the 
action are polar. 
 In \cite{LT} an isometric action of $G$ on $M$  with this property  has 
been called \emph{infinitesimally polar}.    Thus an orthogonal representation is infinitesimally polar in the sense of this paper if and only if the corresponding action on the unit sphere is infinitesimally polar 
in the sense of \cite{LT}.

An isometric action of a compact Lie group $G$ on a Riemannian manifold  
$M$ is infinitesimally polar if and only it  is the case for the 
restriction of the  action to  the 
identity component $G^0$ of $G$.  In such a case the 
canonical projection
$M/G^0 \to M/G$ is an orbi-covering.

Assuming again that  the quotient
$X=M/G$ is a Riemannian orbifold, all singular orbits of the action are 
contained in   boundary of $X$.  Moreover, if $\pi_1  (M)=1$ and $G$ is 
connected then the 
boundary of $X$ is exactly the set
of singular $G$-orbits, while its complement is the union of principal and 
exceptional orbits~\cite{Lyt}. 
In particular, $X$ has no boundary if and only the action of $G$  has no singular orbits. 

Assume now that $M$ has positive curvature.  If the principal isotropy group of the action of $G$ on $M$ is non-trivial then  the quotient $X=M/G$ has non-empty boundary (\cite{wilking-annals}, Lemma 5 and subsequent lines). Hence, if $X$ is a Riemannian orbifold and has empty  boundary, the action of $G$ on $M$ must be almost free. 

 Recall now that any action of a $k$-dimensional torus on a compact positively 
curved manifold
always has at least one orbit of dimension at most $1$, thus with an isotropy group of rank at least $(k-1)$~\cite[Lemma~6.1]{wilking-acta}. Hence:

\begin{lem} \label{posrank}
Let a compact Lie group $G$ act by 
isometries on a compact positively curved Riemannian manifold~$M$. 
If the quotient space $M/G$ has empty boundary and it is isometric
to a Riemannian orbifold of positive dimension   
then $G$ has rank $1$ and the action is almost free.
\end{lem}

\subsection{Quotient-geodesics}\label{quot-geod} 
Let  $M$ be  a connected complete Riemannian manifold on which a compact 
Lie group~$G$ acts by isometries with quotient space  $X=M/G$. 
A geodesic in $M$ is called $G$-\emph{horizontal}, if it is orthogonal 
to the $G$-orbits it meets.
A geodesic is horizontal if and only if it starts in a horizontal direction. 
The projection to $X$ of a horizontal geodesic is a concatenation of 
metric
geodesics (i.e.~locally distance minimizing curves) in $X$.  
We call such a projection a \emph{quotient-geodesic}.  
We refer to \cite[Section~4]{LT} for more on this subject
and for proofs of further statements below.

The projections of two $G$-horizontal geodesics coincide  
if they coincide initially. Hence, we can uniquely  extend an 
initial piece of a quotient-geodesic to a quotient-geodesic defined on $\R$. 
Since any metric geodesic in $X$
is a quotient-geodesic in the above sense, we thus can uniquely extend 
any geodesic to an infinite  quotient-geodesic. 
The set of quotient-geodesics is closed under point-wise convergence.

\begin{rem}  In fact, any quotient-geodesic 
is a so-called \emph{quasi-geodesic} of the Alexandrov space $X$ 
(\cite[Section~5]{Petruninsemi}), a notion which will not play a role below.
Unlike general   quasi-geodesics, the quotient geodesics are uniquely defined 
by their initial segments.  We have a well defined continuous 
quotient-geodesic flow on $X$. 
\end{rem}

There is a dense subset in the space of quotient-geodesics that 
run only through
the principal stratum and strata of codimension $1$.    
This union $X_1$ of all strata
of codimension at most $1$ is a Riemannian orbifold. 
Moreover, the quotient-geodesics
contained in $X_1$  coincide  
with the orbifold-geodesics in the Riemannian orbifold $X_1$.  
By the above  density claim,  the quotient-geodesics
in $X$ are limits of orbifold-geodesics in $X_1$, hence they are 
defined only in metric terms, independently of the presentation of 
the metric space $X$ as a quotient $X=M/G$. 
In particular, if $X$ is a Riemannian orbifold then the 
quotient-geodesics are exactly the 
orbifold-geodesics of $X$.

On any $G$-horizontal geodesic  $\gamma $, all but discretely many points have the same isotropy group $H =G_\gamma$.  Hence the projection of $\gamma$ to $X$ is completely contained in the closure of the $H$-isotropy stratum of $X$.   
Therefore, the closure  of any isotropy stratum is \emph{totally quotient-geodesic} in $X$, in the sense that any quotient geodesic 
contained in the closure of the stratum initially, is contained in it for 
the whole time.

We would like to mention that the above observation  about the dependence  of the quotient-geodesics only on the metric structure of the quotient answers Question~4.6(1) in~\cite{AKLM2}.




\subsection{Luna-Richardson-Straume and minimal reductions}\label{reduction}
If $H$ is the principal isotropy group of the action of $G$ on $M$, then 
we have a canonical isometry $F/N\to M/G$, where $F$ is 
the closure of the set of all points of $M$ whose isotropy group is 
exactly $H$, and $N$ is the normalizer of $H$ in 
$G$~\cite{lr,str,grove-searle,GL}.  The action
of $N$ on $F$ has $H$ in its kernel; the induced action 
of $N/H$ on $F$ is called the \emph{principal isotropy reduction} 
of the action of $G$ on $M$.
It has trivial principal isotropy groups and the same quotient 
space as the original action. 

In case of a linear representation of a compact Lie group $G$ on a 
vector space $V$, the subset $F$ of $V$ is a subspace which we denote 
by $W$. Some properties of such and more general reductions (see below)  have been  investigated in \cite{GL}  to which we refer the reader 
for details. We will use the following special case of 
the results therein.

\begin{lem}\label{conn-irred-reduction}
Let $\rho:G\to\mathbf O(V)$ be a representation of a 
compact Lie group $G$. Assume $G$ is connected and its action on $V$ is 
irreducible, non-polar, infinitesimally polar and has cohomogeneity $c$
at least~$4$. Let $H$, $N$, $W$ 
be as above. Then the action of the identity component
$(N/H)^0$ on $W$  is irreducible. 
\end{lem}

\Pf Suppose, to the contrary, that the action of $(N/H)^0$ on $W$ 
is reducible. Since $G$ is connected and $V/G=W/N$, Theorem~1.7
in~\cite{GL} implies that $(N/H)^0$ is a torus $T^{c-2}$ and its action on 
$W$ is equivalent to the action of the maximal torus of $\SU{c-1}$
on $\C^{c-1}$. However, this action is not infinitesimally polar
for $c\geq4$~\cite[Lemma~7.1]{GL}, and this is a contradiction. \EPf

\medskip

In a very last step of the proof of main theorems we will  need to go through some lists of representations. To simplify the arguments excluding those candidates  one  can make use 
of the following slightly more general minimal reductions studied in \cite{GL}.  For a 
representation $\rho:G\to \mathbf O(V)$, a \emph{minimal reduction} of $\rho$ 
is  a
representation $\tau: K\to\mathbf O(U)$ with the same orbit
space, $V/G=U/K$, and lowest possible dimension of $K$.   It often happens
that the principal isotropy reduction is already a minimal reduction,
but there are easy examples in which it is not the case~\cite{GL}.

Note
that the representation $\rho$ is 
infinitesimally polar if and only if so is its minimal reduction $\tau$.
Moreover,  \lref{conn-irred-reduction} applies to this situation as well, 
showing in particular that if $G$ is connected and its action on $V$ 
is irreducible and infinitesimally polar then 
its minimal reduction group $K$ cannot have a non-trivial
toric connected component, unless the cohomogeneity of $\rho$ is at most $3$.



\subsection{Infinitesimally polar
representations with trivial principal isotropy groups}

The principal isotropy/minimal  reduction allows us to reduce some questions to the case of actions with trivial
principal isotropy groups.  In such a case,  
all slice representations must have trivial principal isotropy groups as well.  On the other hand, if the quotient space is Riemannian orbifold then all 
isotropy representations are polar. But polar representations with 
trivial principal isotropy groups are exceedingly rare:


\begin{lem}\label{slice}
Let $\rho$ be an 
 effective polar representation of a 
compact connected Lie group $G$ on $V$. Assume that the 
principal isotropy group of the representation is trivial. Then  
$G$ is finitely covered  by  $\U 1 ^{l_1} \times \SP 1 ^{l_2}$  and $\rho$ is orbit equivalent to 
$(l_1 \cdot\C \oplus l_2 \cdot \mathbb H \oplus  l_3 \cdot \R )$, where $\R$
denotes the trivial representation.  
Moreover there exists a connected 
normal subgroup
$H$ of $G$ of 
rank $1$ which is the isotropy group 
corresponding to a codimension $1$-stratum of the quotient space $V/G =
(\R ^+ ) ^{l_1+l_2} \times \R^{l_3}$. 
Finally, either $G$ is a torus, or $H$ can be chosen to be 
isomorphic to $\SP 1$.  
\end{lem}  

\Pf We prove the first assertion by induction on the number
of irreducible components of $V$. Without loss of generality, 
we may assume that $V$ has no trivial components. Consider first the case
in which $V$ is irreducible. A quick enumeration of polar irreducible
representations of connected groups yields that $(G,V)$ is equal 
to~$(\U1,\C)$ or~$(\SP1,\Q)$. In fact, we have only to examine the 
list of isotropy representations of compact irreducible symmetric spaces of 
maximal rank, which is a short list, and the list of additional
polar irreducible representations in~\cite{EH}, which is even shorter.  

Assume next $V=V_1\oplus V_2$ is a decomposition into $G$-invariant
subspaces. Choose a regular point $v_i\in V_i$ for the restriction
$(G,V_i)$ for $i=1$, $2$. Owing to~\cite[Theorem~4]{D}, 
$(G_{v_2}^0,V_1)$ (resp. $(G_{v_1}^0,V_2)$) is orbit equivalent to 
$(G,V_1)$ (resp. $(G,V_2)$). The assumption of triviality of principal
isotropy groups for $(G,V)$ implies that $(G_{v_2}^0,V_1)$ (resp. $(G_{v_1}^0,V_2)$) has the same property and it is effective;
of course, it is also polar. It follows
from the inductive hypothesis that
$G_{v_2}^0\times G_{v_1}^0$ is covered by 
$\U1^{l_1}\times\SP1^{l_2}$ and its action
on $V$ is orbit equivalent to $l_1\cdot\C\oplus\l_2\cdot\Q$. 
Again by~\cite[Theorem~4]{D}, $G=G_{v_2}^0\cdot G_{v_1}^0=G_{v_1}^0\cdot G_{v_2}^0$ and $(G,V)$ is orbit equivalent to 
$(G_{v_2}^0,V_1)\times(G_{v_1}^0,V_2)$. Since $(G,V)$ has trivial principal
isotropy groups, $G_{v_1}^0\cap G_{v_2}^0=\{1\}$, showing that 
$G$ is (finitely) covered by $G_{v_1}^0\times G_{v_2}^0$.

Note that the action of  $\tilde G:=\U1^{l_1}\times\SP1^{l_2}$ on $V=l_1\cdot\C\oplus\l_2\cdot\Q$
needs not be a direct product of representations.   
However, $G$ acts on each $\C$-summand as $\U 1$; 
in particular $\SP1^{l_2}$ acts trivially on $l_1 \cdot \C$. 
On the other hand, $G$ acts 
on each $\Q$-summand of~$V$ as one of $\SP1$, $\U2$, $\SO4$.


We next prove the second assertion in the statement.   
The claim is that we can find a $\U1$- or $\SP1$-normal subgroup 
$H$ of $G$ which fixes all 
summands but one 
in the decomposition 
$V=l_1 \cdot\C \oplus l_2 \cdot \mathbb H \oplus  l_3 \cdot \R $.
Again we may assume $l_3=0$.  If $l_2=0$ then $\rho (G)$ is  
the maximal torus of $\U {l_1}$, the representation can be written as
the $l_1$-fold direct product of the representation of $\U 1$ on $\C$ 
and we can take $H$ to be any of these $\U 1$-factors. Consider now  $l_2 >0$, let $G_1,\ldots,G_{l_2}$ denote 
the $\SP1$-factors of $G$ and let $V_1,\ldots,V_{l_2}$ denote the $\Q$-summands 
of $V$. Each $G_i$ can act non-trivially only on the $\mathbb H$-summands.  
Thus assuming, contrarily to the claim, that no such factor corresponds to a 
codimension one stratum, we see that each $G_i$ acts on at least two $V_j$'s.
However, on each $V_j$ at most two of the $G_i$'s can act non-trivially.  
Thus each $G_i$ acts non-trivially exactly on two $V_j$'s. 
By relabeling, we may assume that there exists an integer $m$ such that 
$2\leq m\leq l_2$ and $G_i$ acts non-trivially on $V_i$, $V_{i+1}$ for
$i=1,\ldots,m-1$ and $G_m$ acts on $V_m$, $V_1$.  
Now the action of $G'=G_1 \times \cdots \times G_m$
on $V'=V_1\oplus\cdots\oplus V_m$ is a direct factor of the action of 
$G$ on $V$. 
However, $(G',V')$ has principal isotropy subgroup $\SO2$, given by a maximal 
torus in the diagonal $\SP1$-subgroup of $G'$, and this contradicts the 
fact that the original representation has trivial principal isotropy groups. 
\EPf
\medskip


The next lemma will allow us to apply Lemma~\ref{slice} to certain
slice representations.

\begin{lem}\label{iso-conn}
Let a connected Lie group $G$ act by isometries on a simply connected Riemannian manifold $M$. Assume that  the principal isotropy groups of $(G,M)$ are trivial. If the quotient $M/G$ is a good Riemannian orbifold then all isotropy groups of  $(G,M)$ are connected. 
\end{lem}

\begin{proof}
From  \cite[Theorem 1.8]{Lyt} we know that all slice representations have connected orbits. Since the slice representations have trivial principal isotropy groups, all isotropy groups must be connected.
\end{proof}




\section{Discussion of the examples} \label{secdiscuss}
In  this section we discuss the representations of compact connected groups $G$ appearing  in  \tref{thmsum}.

\subsection{Polar representations} They are listed in \cite[tables~8.11.2
and~8.11.5]{Wolf}
and~\cite{EH,B,fgt}, and their 
properties are well known  by now.  These are exactly the representations 
whose spherical quotient has constant curvature $1$.

\subsection{Almost free actions}\label{subsecfree} We call a representation
$\rho$ of the group $G$ almost free, if $G$ acts almost freely on the corresponding unit sphere. In this case  $G$ has rank one (\lref{posrank}), hence it is either $\U1$, or it is covered 
by $\SU 2$.

If $G=\U 1$ then any almost free representation is the sum of non-trivial 
 irreducible ones which are 
all complex one-dimensional. Such irreducible representations 
are  naturally parametrized by positive  integers  $r$, 
the order of the kernel. 
The quotient space $X^{2l}=\mathbb S^{2l+1} /\rho (\U 1 )$    is called a 
\emph{weighted complex projective space}, with weights $r_1,\ldots,r_{l+1}$,
where $r_1,\ldots,r_{l+1}$ are the integers associated to each one of the summands 
of $\rho$.

If $G$ is covered by $\SU 2$ then all almost free representations are described similarly 
(cf.~\cite{gromollsph}).  It turns out that any almost free representation 
$\rho$ of $G$ is a sum of irreducible representations of quaternionic type
and $G=\SU2$.
For any natural number $r\geq 0$ there exists exactly one such irreducible 
representation of complex dimension~$2r$; the  almost free representations of $\SU 2$ are arbitrary 
sums of such representations.    
It seems natural to call the quotient space $X ^{4l}=\mathbb S^{4l+3} /\rho(\SU 2)$
a \emph{weighted quaternionic projective space}, also in view of the following 
remark.  

\begin{rem} 
The topology of the weighted projective spaces is very close to the topology of the usual projective spaces. In both cases the orbifold cohomology of the quotient $X$ 
can be computed
using the fibration $G\to \bS^n \to \hat X$, where $\hat X$ is 
Haefliger's classifying space
of the orbifold $X$. The cohomology of $\hat X$, which by 
definition coincides with the orbifold cohomology of $X$, 
is generated by one element $e$ in degree $2$ or $4$, respectively.  
The only relation is $p\cdot e^k=0$ for some natural number $p$. This number 
$p$ is $1$ if and only if the representation acts freely on the 
sphere, i.e., if and only if  the  representation is a Hopf action. 
In any case $X$ has the same rational homotopy type as the corresponding 
projective space.
\end{rem}

All weighted projective spaces have trivial orbifold fundamental group.  Thus they are good orbifolds if and only if they are Riemannian manifolds. This happens if and only if the action of $G$ on $\bS^n$ is free, i.e., if and only if the quotient $X$ is a classical projective space.    
Thus $X$ is diffeomorphic to a sphere if and only if $l=1$, in which case it has constant curvature $4$.

\subsection{Non-polar representations of cohomogeneity $3$}\label{subsecstraume}
This case is studied in detail in \cite{str}. There  it is shown that any such 
representation has a reduction to a representation of a one-dimensional group 
on $\R^4$.    In fact, if $G \neq \U 1$ then the reduction is either a finite extension
of the Hopf action or a two-fold extension of the action of $\U 1 $ on 
$\C\oplus\C=\R^4$ with
parameters $(1,2)$.  In the second case, the quotient space $X$ is a disc with one singularity with 
angle $\pi/2$.   
In any case,
if $X$ is a good orbifold, then it has constant curvature $4$.

\subsection{Complex case}\label{cx-case}  
We discuss here the  second and the third examples from  Table 1
(see also \cite[Prop.~5.2]{Radeschi2} where the same result is proven from a very different point of view).
First we consider the double of the vector representation of $\U 2$.
We view $\bS^7\subset\Q^2$ and $\U2$ as left multiplication by a unit quaternion
and right multiplication by a unit complex number.  
To obtain $\mathbb S^7 /\U 2$ we first divide out $\SU 2$ to obtain  
$\Q P ^1$, the round sphere of curvature $4$ and then
we have to divide out the remaining isometric action of $\U 1 =\U 2 /\SU 2$  on $\Q P ^1$.

In homogeneous coordinates 
$(q_0:q_1)e^{i\theta}=(q_0e^{i\theta}:q_1e^{i\theta})\in\Q P^1$
corresponds to $e^{-i\theta}q_0^{-1}q_1e^{i\theta}\in\Q\cup\{\infty\}=\bS^4(\frac12)$ 
and so $e^{i\theta}\in\U1$ fixes the plane spanned
by $1$, $i$ and rotates the plane spanned by $j$, $k$ by $2\theta$.
Hence the action
of $\U 1$ on $\Q P ^1$ is polar and the quotient is a $3$-dimensional 
hemisphere $\bS^3_+(\frac12)$ of constant curvature $4$.

The  group $\SU n$ acting on $\C^n\oplus\C^n$ is orbit-equivalent to
$\U n$ for $n\geq3$. Moreover, 
we can apply the reduction procedure to $\U n$ 
and obtain $(\U2,\C^2\oplus\C^2)$
as a minimal reduction. Hence the second and third
representations in Table~1 yield 
a quotient isometric to~$\bS^3_+(\frac12)$.

\subsection{Quaternionic case}\label{cohom6}
Herein we deal with the four last representations
in Table~1; again a different approach can be found in 
\cite[Prop.~5.2]{Radeschi2}. 

Consider first $\SP2$ acting on $V=\Q^2\oplus\Q^2$. 
It is a reduction
of $(\SP n,\Q^n\oplus\Q^n)$ for $n\geq3$, so 
the quotient
$X=\mathbb S^{8n-1} /\SP n=\mathbb S^{15} /\SP 2$ has a non-empty boundary. 
The representation has trivial principal isotropy group, hence 
$X$ has dimension $5$.  
The normalizer of $\SP 2$ in 
$\mathbf O(V)$
equals $\SP2\cdot\SP2$.  Thus we have a non-trivial  
action of $\SP 2$ on $X$ which 
leaves the boundary invariant, fixes the unique 
point $x$ of maximal distance to $\partial X$, and 
acts as $\SP2/\Z_2=\SO5$ on the tangent space at $x$. 
The quotient of the $\SP2$-action on $X$ is given by 
$\bS^{15} /\SP 2 \cdot \SP2$.
This space is known to be an interval of length~$[0,\frac\pi4]$,
which says that $\SO 5$ acts on $X$ with cohomogeneity one.
Now we see that $X$ must be a  hemi-sphere 
with a rotationally invariant metric; we are going 
to show that it has constant curvature $4$.

Let $\gamma:[0,\frac\pi4]\to V$ be the horizontal 
geodesic   
\[ \gamma(r)  = \left(\begin{array}{cc}\cos r&0\\0&\sin r\end{array}\right) \]
Its projection under $\pi:\bS^{15}\to X$ 
starts at a boundary point of $X$ and ends at~$x$.
Note that $\pi^{-1}(\partial X)$ consists of pairs of linearly dependent
vectors over $\Q$, $\pi\left(\left\{\left(\begin{array}{cc}*&*\\0&0\end{array}\right)\right\}
\right)=\partial X$, and
\[ \partial X=\left\{\left(\begin{array}{cc}*&*\\0&0\end{array}\right)\right\}\Big/
\left(\begin{array}{cc}\SP1&0\\0&1\end{array}\right)=\bS^7/\SP1=\Q P^1
=\bS^4(1/2).\]

The $\SO5$-orbit through $\pi(\gamma(r))$ for
$0< r < \frac\pi4$ is a $4$-sphere of constant curvature.
In fact, due to the presence of $\SO 5$-symmetry the metric on $X$ has the form
$dr ^2 + f^2(r)\, \omega $, where $\omega$ is the metric of constant curvature $4$
on $\mathbb S^4$ and $f$ is a non-negative function. 
Moreover, $f(r)$ is equal to the quotient 
$||\pi_*\xi||_{\pi(\gamma (r))} / ||\pi_*\xi ||_{\pi(\gamma (0))}$, 
for an arbitrary Killing field $\xi$ on $\bS^{15}$ induced by the
Lie algebra $\mathfrak{sp}(2)$ acting on the right 
whose $\pi$-horizontal component does not vanish at $\gamma (0)$.
Now we consider the unit Killing field $\xi$ induced by 
\[ \left(\begin{array}{cc}0&1\\-1&0\end{array}\right) \]
(quaternionic matrix) multiplying $\gamma(r)$
on the right. This Killing field 
is not $\pi$-horizontal with respect to~$\pi$. 
Its $\pi$-vertical component is parallel to the unit 
Killing field $\eta$ on $\bS^{15}$ 
induced by the same matrix 
multiplying $\gamma(r)$ on the \emph{left}. Taking scalar products, the 
$\pi$-vertical
component of $\xi$ has length $\sin(2r)$ 
and the $\pi$-horizontal one, $\cos(2r)$. It follows that $f(r)= \cos (2r)$.
Hence we see that $X$ is a hemi-sphere $\bS^5_+(1/2)$ of constant 
curvature~$4$.  


\medskip

Similarly, the last three representations in Table~1  have reductions 
to the respective representations with $n=2$, so we may assume 
$n=2$ in the discussion of the remaining orbit spaces. In turn, these orbit 
spaces are obtained by dividing $X$ by a subgroup of $\SP2$ acting on $X$
whose action is induced by right multiplication on $V$.  
Namely, we consider $\SP1\times\SP1$, $\Delta_{\mathbf{Sp}(1)}$ (diagonal), 
$\mathbf T^2$, and $\Delta_{\mathbf T^1}=\U1$; here the first group is included
only for the sake of completeness, whereas the remaining three
yield the representations of interest to us. These groups
fix $x$, so it is enough 
to understand the action on $T_xX$: this is given by conjugation. 
Note that $T_xX\cong\mathcal H_x=\langle e_0,e_1,e_2,e_3,e_4\rangle$ 
where 
\[ e_0=\left(\begin{array}{cc}1&0\\0&-1\end{array}\right),\ 
e_1=\left(\begin{array}{cc}0&1\\1&0\end{array}\right),\ 
e_2=\left(\begin{array}{cc}0&-i\\i&0\end{array}\right),\ 
e_3=\left(\begin{array}{cc}0&-j\\j&0\end{array}\right),\ 
e_4=\left(\begin{array}{cc}0&-k\\k&0\end{array}\right). \]
Let $a$, $b\in\Q$. Then:
\[ \U1:\ \left(\begin{array}{cc}e^{-i\theta}&0\\0&e^{-i\theta}\end{array}\right)
\left(\begin{array}{cc}0&a\\b&0\end{array}\right)
\left(\begin{array}{cc}e^{i\theta}&0\\0&e^{i\theta}\end{array}\right)
=\left(\begin{array}{cc}0&e^{-i\theta}ae^{i\theta}\\e^{-i\theta}be^{i\theta}&0\end{array}\right) \]
fixes $\langle e_0,e_1,e_2\rangle$ and acts as $\SO2$ on  
$\langle e_3,e_4\rangle$;
\[ \mathbf T^2:\ \left(\begin{array}{cc}e^{-i\theta}&0\\0&e^{-i\varphi}\end{array}\right)
\left(\begin{array}{cc}0&a\\b&0\end{array}\right)
\left(\begin{array}{cc}e^{i\theta}&0\\0&e^{i\varphi}\end{array}\right)
=\left(\begin{array}{cc}0&e^{-i\theta}ae^{i\varphi}\\e^{-i\varphi}be^{i\theta}&0\end{array}\right) \]
fixes $\langle e_0\rangle$ and acts as $\SO2\times\SO2$ on  
$\langle e_1,e_2\rangle\oplus\langle e_3,e_4\rangle$;
\[ \SP1:\ \left(\begin{array}{cc}q^{-1}&0\\0&q^{-1}\end{array}\right)
\left(\begin{array}{cc}0&a\\b&0\end{array}\right)
\left(\begin{array}{cc}q&0\\0&q\end{array}\right)
=\left(\begin{array}{cc}0&q^{-1}aq\\q^{-1}bq&0\end{array}\right) \]
fixes $\langle e_0,e_1\rangle$ and acts as $\SO3$ on  
$\langle e_2,e_3,e_4\rangle$;
\[ \SP1\times\SP1:\ \left(\begin{array}{cc}q_1^{-1}&0\\0&q_2^{-1}\end{array}\right)
\left(\begin{array}{cc}0&a\\b&0\end{array}\right)
\left(\begin{array}{cc}q_1&0\\0&q_2\end{array}\right)
=\left(\begin{array}{cc}0&q_1^{-1}aq_2\\q_2^{-1}bq_1&0\end{array}\right) \]
fixes $\langle e_0\rangle$ and acts as $\SO4$ on  
$\langle e_1,e_2,e_3,e_4\rangle$.

It follows that the quotient of $X$ by those groups is respectively
\[ \bS^4_{++}(1/2),\ \bS^3_{+++}(1/2),\ \bS^3_{++}(1/2),\ \bS^2_{++}(1/2) \]
(cf.~\cite[Table~II, type~II]{str} in the last case).

\subsection{Exceptional case}
Finally, we deal with the action of $G=\Spin9$ on $V=\R^{16}\oplus\R^{16}$. 

Let $H$ be the principal isotropy group, $W=V^H$ its fixed point
set. The isotropy subgroup at a point in $\R^{16}\oplus\{0\}$ is $\Spin7$. 
The action of $\Spin7$ on $\{0\}\oplus\R^{16}$ decomposes as
$\R\oplus\R^7\oplus\R^8$, the sum of the vector and spin representations
plus a trivial component. This is a slice representation which is polar,
and corresponds to a unique non-principal minimal orbit type. 
It follows that our representation is infinitesimally polar. 
It also follows that $H\cong\SU3$ and thus the dimension 
of $X=\mathbb S^{31}/G$ is $3$. 

The identity component of the normalizer $N_G(H)^0=H\cdot Z_G(H)^0$,
where $Z_G(H)$ denotes the centralizer of $H$ in $G$. One computes
that $Z_G(H)^0\cong\U2$ (for instance, by using root systems;
alternatively compare~\cite[p.~137]{Gor2}).
Let $C:=(N_G(H)/H)^0$. Then  $C\cong\U2$ and $\dim W=8$. 
Since~$(G,V)$ is a doubling representation, so is $(C,W)$;
hence $W=\C^2\oplus\C^2$ (one can also refer to
Proposition~\ref{red} below or \cite{Gor2}).

Now $X'=W/C=\bS^3_+(\frac12)$ due to subsection~\ref{cx-case}, 
and $X=\mathbb S^{31} /G $ is a finite quotient of $X'=W/C$ by a 
group $\Gamma$ generated by reflections~\cite[Prop.~1.2]{GL}.
Note that $\Gamma$ is non-trivial as 
$(G,V)$ has more than two orbit types.

There is an isometric $\SO2$ action
on $X$ with quotient $X/\SO2 =\bS^2(\frac 1 2)_{+++}$.
Indeed there is a circle in the centralizer of $\Spin9$ in $\mathbf O(V)$,
and $\mathbb S^{31}/\Spin9\times\SO2$ is a spherical triangle with 
three right angles~\cite[Table~II, type III${}_4$]{str}. 

The isometry group of $\bS^3 (\frac 1 2) _+$ is $\OG3$ acting through 
the canonical action on the two-dimensional boundary sphere.  
The reflections of $X$ are exactly the reflections in $\OG3$. 
If $\Gamma$ has more than two elements, it contains two different 
reflections, and the quotient  of the boundary sphere by $\Gamma$  
does not admit a circle action, hence this case cannot occur.
Hence $\Gamma$ has two elements, the identity and a reflection.
All reflections in   the isometry group of $\bS^3 (\frac 1 2) _+$ are conjugate
and the quotient of $X'$ under the action of any such reflection is 
thus $X=\bS^3 (\frac 1 2) _{++}$.

\begin{rem} 
A conceptual proof that 
the above orbit spaces have constant curvature
(specially in the case of simple groups, see~\cite{Gor2})
is to check that the representations are taut
and infinitesimally polar, and use~\cite[Theorem~3.20]{Wiesen1}
and ~\cite[Corollary~6.12]{TTh2}.  
\end{rem}

\section{Positively curved orbifolds} \label{secorbifold}
\subsection{Coxeter orbifolds} All orbifolds below are considered with some complete Riemannian metric.
Recall that an orbifold is called a \emph{Coxeter orbifold} if all of its local
groups are 
finite Coxeter groups acting as reflection groups on the corresponding tangent spaces.
In \cite{Davis} such orbifolds  are called \emph{reflectofolds}.  We denote by $X_{reg}$ 
the set of manifold  points of an orbifold.

We start with some preliminaries:

\begin{lem} \label{cox}
An orbifold $X$ is a Coxeter orbifold if and only if  $X\setminus X_{reg} =\partial X$. 
\end{lem}

\begin{proof}
 The claim and conditions are local.  Thus we may assume that $X$ has the form
$X=\R ^n /\Gamma$ for a finite group $\Gamma$ of linear isometries.

   The only if direction is clear, due to the structure of Coxeter chambers.  
Assume now that any non-regular point
of $X$ is contained
in $\partial X$.  Let $\Gamma _{refl}$ be the normal subgroup of $\Gamma$ 
generated by reflections.  Thus $\Gamma _{refl}$ is a finite reflection 
group and we need to prove 
$\Gamma =\Gamma _{refl}$. 

 Consider the Coxeter chamber $X'= \R^n /\Gamma _{refl}$ and the induced 
effective, isometric action of $\Gamma ' =\Gamma /\Gamma_{refl}$ on the 
orbifold $X'$. 
The set of regular
points $X'_{reg}$ is homeomorphic to $\R^n$, and it  
is preserved under the action of $\Gamma '$.   Hence, if $\Gamma '$ is 
non-trivial, the action of the finite group $\Gamma '$
on the contractible manifold $X'_{reg}$ cannot be free.  On the other 
hand, $\Gamma'$ has no reflections, thus  the set of fixed points in $X' _{reg}$ of any non-trivial element in $\Gamma'$ has codimension at least two. 
Therefore,  the  image  of this  set of fixed points   in $X$ consists 
of singular points outside of $\partial X$.

 We deduce that this set of fixed points is empty, hence the action 
of $\Gamma '$ on $X'_{reg}$ is free, hence $\Gamma'$ is the trivial 
group and $\Gamma =\Gamma _{refl}$. 
\end{proof}

The above lemma  is closely related to the following result which we 
quote  from \cite[Theorem~1.8]{Lyt}:

\begin{lem} \label{noex}
Let the connected group $G$ act by isometries on a simply connected Riemannian manifold $M$.   Assume that the quotient $X=M/G$ is a Riemannian orbifold.
The orbifold $X$ is a Coxeter orbifold if and only if $X\setminus \partial X$ is a 
good orbifold.  In particular, if $X$ is a good orbifold, it is automatically a Coxeter
orbifold.
\end{lem}

For an orbifold $X$ consider the universal 
covering $\hat  X$ of 
the underlying topological space with the induced orbifold structure.  
Since being a Coxeter orbifold is a local condition, 
$X$ is a Coxeter orbifold if and only if  $\hat X$ is a Coxeter orbifold.
 The universal orbi-covering $\tilde X$ is also the universal orbi-covering of $\hat X$.
Hence investigating the question, whether a given Coxeter orbifold $X$ is a good orbifold we may restrict to the case, when the  underlying topological space $X$ is simply connected.

  For a Coxeter orbifold $X$,  the \emph{faces} are by definition the strata of 
codimension $1$.     The closure of a face is called a \emph{mirror}. 
Then each singular point of $X$ is contained in some mirror.  Given any   point $x\in X$, consider a small open convex ball
$U$ around $x$.    Since $U$ is diffeomorphic as an orbifold to a Coxeter chamber,
we see that $x$ is  contained in exactly $k$ mirrors  of the orbifold
$U$  if and only if $x$ lies on a stratum of codimension $k$.
  Note that a face of $U$ is part of a face of $X$, but that different faces of $U$ could be part of the same face of $X$.

It is  easy to characterize which Coxeter orbifolds  with simply connected underlying space are good, expressed in terms of the intersections of mirrors. 
One part of the following result is straightforward  
 (\cite{Davis}, Section 5.1); the other one is contained in 
\cite{AKLM}, Corollary 6.5.  (In \cite{AKLM},  the result is unfortunately incorrectly stated without the assumptions (C1), (C2) below. However, the assumptions are implicitly made
in Section 4 of \cite{AKLM}, where they are wrongly assumed to hold always.)   For convenience of the reader, we recall  the proof of this result:

\begin{lem}  \label{goodness}
Let $X$ be a Coxeter orbifold with $\pi_1 (X)=1$.   The orbifold $X$ is good if and only if it satisfies the following  conditions:

(C1) Any codimension two stratum is contained in two  \emph{different} mirrors. 

(C2)  If the intersection $\bar  W \cap \bar{ W'} $ of two different 
mirrors contains two different strata $B^{\pm}$  of codimension $2$, then   
the orders of the isotropy groups    at $B^+$ and at $B^-$  coincide.
\end{lem}

\begin{proof}


The orbifold fundamental group $\Gamma $ of $X$ 
is generated by
involutions  $s_W$, indexed by the faces $W$ of $X$ (cf. \cite{Davis}).   The relations are generated
by words of the form $r_S$, indexed by strata $S$ of codimension $2$.  
For any such stratum $S$, consider a point $x\in S$ and  a small neighborhood $U$ of $x$.  Then the word $r_S$ has the form $r_S= (s_W \cdot s_{W'} ) ^n$,
where $n$ is the dihedral angle at  $x$ (half of the order of the local isotropy group $\Gamma _x$)   and where $W$ and $W'$ are the faces that extend the local faces in $U$  adjacent to $x$.

For any $x\in X$ the local group $\Gamma _x$ is the finite Coxeter group generated by the involutions indexed by the local faces at $x$. And the canonical map $\Gamma _x \to \Gamma$ sends the generator of $\Gamma _x$ corresponding to a local face to the generator of $\Gamma $ corresponding to the global face which extends the local one.

  The orbifold $X$ is good if and only if all maps $\Gamma _x \to \Gamma $ are injective.
If one of conditions~(C1) or~(C2) does not hold, then the map  from $\Gamma _x$ to $\Gamma$ is not injective, where $x$ is any point in a codimension $2$ stratum appearing in (C1), respectively, in the stratum among~$B^+$, $B^-$
at which the local isotropy group is larger in case~(C2).

On the other hand, if~(C1) and~(C2) hold true then $\Gamma$ is the 
Coxeter group with the Coxeter generating system $\langle s_W |r_S \rangle$, 
and any local group $\Gamma _x$ is given by a Coxeter subsystem.  
The statement that $\Gamma _x \to \Gamma$ is injective
is a standard statement about Coxeter groups.
\end{proof}

Let again $X$  be a Coxeter orbifold with $\pi_1 (X)=1$. Let $W$ be a face.  Consider $W$ with its intrinsic metric and   let $N$ be the completion. 
 Looking  at the situation locally, we see that $N$ is 
an Alexandrov space, which is in addition a  \emph{Riemannian manifold 
with corners} (cf. \cite{AKLM}). We have a canonical  map 
$i_W:N\to X$, which is an isometric embedding in a neighborhood of any point of $N$.
The image of $i_W$ is the  mirror $\bar W$.    The map $i_W$ is injective on $W$, but, 
a priori, it could happen,
that the map $i_W$ is not injective on $N$.

  The space $N$ considered as a manifold with corners has a canonical stratification, and 
any  stratum  of codimension $k$ is mapped by $i_W$ to a stratum of codimension $k+1$ in  $X$.  The closure of any stratum of $N$
is an extremal subset of $N$ in the sense of Alexandrov geometry.   We consider the following condition on the orbifold $X$:

(C3) For any face $W$,   any pair of  closures of  codimension $1$ strata in the completion $N$ of $W$ have a non-empty intersection in $N$.  

\begin{lem}
Let $X$ be a Coxeter orbifold. The condition (C3) implies  (C1) and (C2).
\end{lem}

\begin{proof}
Assume that (C1) does not hold. Consider a bad stratum $A$ of codimension $2$, an adjacent face $W$ and the completion $N$.    Then we have
two different strata $A^{\pm}$  of codimension $1$ in $N$ that are mapped by
$i_W $ to the same stratum $A$.    By assumption $\bar A^+ \cap  \bar A^-$ is not empty.  At any point $z$ in this intersection, the restriction of $i_W$ to any neighborhood  of $z$ is non-injective, which is impossible. 

Assume that (C2) does not hold and consider $W$, $W'$, $B^{\pm}$ as in the 
definition.
We claim that $\bar B^+ \cap \bar B^-$ is empty.  Otherwise,  
in the neighborhood 
of any point $x$ in the intersection, we would see two local mirrors which intersect in different local strata of codimension $2$.  But this does not happen in a Coxeter chamber.

On the other hand, consider the completion $N$ of $W$ and the two strata $A^{\pm}$ in $N$ that are mapped by $i_W$ to $B^{\pm} $ respectively.   Choose a point  $z$ in the non-empty intersection $\bar A^+ \cap  \bar A^-$, which must exist by assumption (C3).
Then the point    $x=i_W (z)$ is contained in $\bar B ^+ \cap \bar B^-$, in contradiction to the previous argument.
\end{proof}

\subsection{Positively curved Coxeter orbifolds}
Assume now that a
compact Coxeter orbifold $X$ is positively curved.   
The completion $N$ of any face $W$ is positively curved as well.
From the theorem of Petrunin-Frankel
(\cite{Petrunin}, Corollary 3.3) we deduce that any pair of strata of codimension $1$ in $N$   has a non-empty intersection, if the dimension of 
$N$ is at least $2$ and the curvature greater than some positive constant.  Thus 
$X$ satisfies the condition  (C3), if $\dim (X) \geq 3$.  
We can also apply this reasoning to 
the universal covering $\hat X$ of $X$ (since it has curvature bounded   
below by some positive constant, so it is a
compact Coxeter orbifold as well). We arrive at:

\begin{lem}  \label{pet}
Let $X$ be a compact positively curved Coxeter orbifold
of dimension at least $3$. Then $X$ is a good orbifold.
\end{lem}

In the course of the proof of  \tref{good}, we will be needing 
the following result whose proof is contained in the first few lines of the 
proof of Proposition~2.3 in~\cite{fgt}:

\begin{lem}  \label{fgt-lemma}
Let $M$ be a simply connected compact positively curved manifold. 
Let $g:M\to M$ be a reflection at a totally geodesic hypersurface $N$.
Then $M$ and $N$ are diffeomorphic to spheres.
\end{lem}

Now we can prove  \tref{good}:

\begin{proof} [Proof of \tref{good}]
Let $\tilde X$ denote the universal orbi-covering of $X$. By assumption,
$\tilde X$ has uniformly positive curvature, in particular it is compact.
There is at least one reflection~$g$ on $\tilde X$. 
We may assume that $g$ is the reflection at a totally geodesic
hypersurface $Z$, dividing $\tilde X$ in two isometric connected components.  
Without loss of generality, we may assume  
$X=\tilde X/\{ 1,g\} $. Then $X$ has boundary $\partial X=Z$
(as an Alexandrov space) and the orbifold $X_0:=X\setminus\partial X$ 
has trivial orbifold fundamental group (for instance,  
cf.~\cite[Prop.~3.4]{GL}).

We claim first that $X_0$ is a manifold.   
We refer to \cite[pp.5-23]{KleinerLott}  
for details of the following argument. 
By positivity of curvature, the distance function $f=d_Z$ is 
strictly concave on 
$X_0$.  It has a unique critical point $p$ 
(in the sense of Alexandrov geometry, which is 
equivalent to being critical in the orbifold sense).  
Moreover,  $X_0 \setminus \{p\}$
is diffeomorphic  (as an orbifold) to  the product $(0,t) \times T$, 
where $T$ is a level set of $f$.  This implies 
that $\pi_1 ^{orb} (X_0) = \pi_1 ^{orb} (X_0 \setminus \{ p\})=
\pi_1 ^{orb} (U_p)$, for any small neighborhood $U_p$ of $p$. 
Since $\pi_1^{orb}(X_0)=1$, this shows that
 $p$ is a regular point of the orbifold $X_0$.  Since the whole 
orbifold $X_0$ is diffeomorphic to an open suborbifold of $U_p$, we 
deduce that $X_0$ is a manifold.

Due to \lref{cox}, $X$ a Coxeter  orbifold. Due to \lref{pet}
the orbifold $X$ is a good orbifold.  Hence,  $\tilde X$ is a manifold. 

Now, $\tilde X$ is a positively curved manifold with an isometric  
reflection at a hypersurface. By \lref{fgt-lemma} we deduce that  
$\tilde X$ must be diffeomorphic to a sphere.
\end{proof}

\section{Orbifolds isometric to quotients of spheres} \label{secorbiquot}

\subsection{Reductions and immersions of strata} \label{strata}

We consider a more general construction than that in 
subsection~\ref{reduction}. 
Let a compact Lie group $G$ act by isometries on a complete connected 
Riemannian  manifold $M$. Denote the quotient space $M/G$ by~$X$. 
Let $H\subset G$ be the isotropy group of an arbitrary point in $M$. 
Denote the normalizer of $H$ in $G$ by $N$.  
Consider the set $F'$ of fixed points of $H$, and let $F^0 \subset F'$ be 
the set of all points whose isotropy group is exactly $H$.  
Then $F'$ is totally geodesic in $M$, $F^0$ is open in 
$F'$ and the closure $F$ of $F^0$ is a union of some  components of $F'$. 
The sets $F$ and $F^0$ are $N$-invariant.  The embedding 
$F\to M$ defines a canonical map $I_H:F/N\to M/G$.
For a general, non-principal isotropy group $H$,  
the map $I$ needs not be injective, as the following typical example shows
(the version of this map in the category of algebraic actions
of complex reductive groups on affine varieties
is a normalization map, cf.~\cite[\S2]{Lu4}).

\begin{ex} \label{example}
Consider the representation of $G=\U 2$ on $\C^2 \oplus \R ^3$, 
where $G$ acts on $\R^3$ as $\SO3$ 
(Case 12 in Table 3 below; see also \cite[Table~II, type~I]{str}).
  Let $M$ be the unit sphere
with the induced action of $G$. The quotient space $X=M/G$ is 
half of a tear-drop. It is topologically a disc, which has three 
isotropy strata, all of them connected: the open disc, 
the complement of one point on the boundary circle and this one point.
The $1$-dimensional stratum corresponds to the isotropy group $H=\U 1$.  The
corresponding normalizer $N$ is the maximal torus $\U 1\times \U 1$. 
The set $F$ of fixed points is the unit sphere $\mathbb S^2$ in 
$\C \times \R$  and the quotient space $F/N$ is the interval of length $\pi$.  
The map $I$ sends this interval length-preserving onto the boundary of 
the disc and folds both end-points to the  most 
singular point of $X$.
\end{ex}

In general we have:

\begin{lem} \label{immersion}
Under the assumptions above, the map $I_H : F/N \to M/G$ is $1$-Lipschitz, 
finite-to-one
and preserves the lengths of curves. The restriction of
$I_H$ to $F^0/N$ is an injective local isometry 
onto the $H$-isotropy 
stratum of $M/G$.  Moreover, $I_H$ preserves the quotient-geodesic flow, i.e., it sends a projection of a horizontal geodesic 
to the projection of a horizontal geodesic.
\end{lem}

\begin{proof}
The image $I_H (F^0 /N)$  is the set of all $G$-orbits that contain a point with isotropy group $H$, hence exactly the $H$-isotropy stratum of $M/G$.  If two points $p_1$, $p_2 \in F_0$ are in the same $G$-orbit then any  element   $g\in G$ sending $p_1$ to $p_2$ must normalize $H$. Hence, $p_1$ and $p_2$ are in the same $N$-orbit.  Thus the map $I_H$ is injective on $F_0/N$.

The remaining claims follow directly, once we know that $I_H$ preserves the 
quotient-geodesic flow. Since $F_0/N$ is dense in $F/N$ it suffices
to see that an $N$-horizontal geodesic in $F^0$ is a $G$-horizontal geodesic in $M$.  But this becomes clear by looking at the slice  representation (of $H$) at any point of $F^0$.
\end{proof}

\subsection{The orbifold case}
In case that the quotient space is an orbifold we deduce:
\begin{cor}
Let a compact Lie group $G$ act by isometries on a complete 
Riemannian manifold $M$. Assume that $X=M/G$ is a Riemannian orbifold. 
Let $H$ be any isotropy group of the action, with normalizer $N$ and $F$ 
defined as above.  Then $F/N$ is a Riemannian orbifold.  
The map $I_H:F/N\to M/G$ is a totally geodesic isometric immersion
of orbifolds.
\end{cor}

\begin{proof}
The subset $F^0 /N$ is open and dense in $F/N$.  
On the other hand, $I_H$ is a Riemannian isometry 
from $F^0 /N$ onto its image.
This image is a stratum in a Riemannian orbifold.
We deduce that the sectional curvatures in the principal part of 
$F^0/N$  are uniformly bounded from both sides. Hence $F/N$ 
is a Riemannian orbifold by~\cite{LT}. 

The second claim follows from the first claim, \lref{immersion}   
and the fact that the quotient-geodesic flows in  $X$ and  in $F/N$ coincide  with the respective orbifold-geodesic flows (Subsection~\ref{quot-geod}).
\end{proof}

\subsection{Spherical case}
Now we state the application of the strata immersion map $I_H$ which will be used in the  inductive proof of our main step:

\begin{cor}  \label{induction}
Let the compact Lie 
group $G$ act isometrically on the round sphere $\mathbb S^n$.
Assume that the action is non-polar and that the quotient $X=\mathbb S^n /G$
is a Riemannian orbifold  of dimension $k\geq 3$ with a  non-empty boundary. 
Then there exists a non-polar  isometric  action of a 
compact Lie group $N$ on  some unit sphere $\mathbb S ^m$, 
such that the quotient $Y=\mathbb S^m /N$ is a good Riemannian orbifold
of dimension~$k-1$.
Moreover, there is a isometric reflection $r$ on the universal 
orbi-covering $\tilde X$ of $X$, whose set $Z$ of fixed points is the universal orbi-covering    of $Y$. In  particular, the universal orbi-covering of 
$Y$ is diffeomorphic to a sphere.
%
%
%
\end{cor}

\begin{proof}
Recall that $X$ is a good Riemannian orbifold and $\tilde X$ is 
diffeomorphic to a sphere by Theorem~\ref{good}.

Choose a point $p\in \mathbb S^n$  which is projected to a codimension one 
stratum $St$ in $X$.
Let $H$ be the isotropy group of $p$. Then the set of fixed points $F$ of 
$H$ is 
a great subsphere of $\mathbb S^n$ of some dimension $m$. The results 
above show that $Y=\mathbb S^m /N$ is Riemannian orbifold and that $I_H:Y\to X$ is an isometric (totally geodesic) immersion whose image   contains $St$.  In particular,
$Y$ has dimension $k-1$.

The map $I_H$ lifts to an isometric, totally geodesic immersion $\tilde  I_H :Z\to 
\tilde X$, between the universal orbi-coverings 
of $Y$ and $X$, respectively~\cite[p.~611]{BrH}.  Since $\tilde X$ is a manifold, we deduce that $Z$ is a manifold. In particular, $Y$ is a good Riemannian orbifold.

Let $x\in \tilde X$ be any point projected to $St$ in $x$.  Then there is a isometric reflection $r:\tilde X\to \tilde X$  (belonging to~
 $\pi _1 ^{orb} (X)$), whose set  $Z'$ of fixed points is a $(k-1)$-dimensional totally geodesic submanifold containing  $x$, whose projection to $X$ contains $St$.    We see by dimensional reasons and the fact that $\tilde I_H$ is totally geodesic, that the image $\tilde I_H(Z)$ coincides with $Z'$.  The map
$\tilde I_H:Z\to Z'$ is a local isometry, hence a covering map. By \lref{fgt-lemma}, $Z'$ is diffeomorphic to a sphere. Since $k\geq 3$, the sphere $Z'$ is simply connected and
$\tilde I_H$ is an isometry.

It only remains to prove that the action of $N$ on $\mathbb S^m$ is non-polar.
Otherwise, $Y$ and hence $Z'$  had constant curvature $1$. The submanifold $Z'$ divides the manifold $\tilde X$ into two submanifolds with totally geodesic boundary $Z'$.
Since the curvature of $\tilde X$ is at least $1$, we can apply a special case of a theorem by Hang and Wang~\cite[Theorem~2]{hangwang} to see that $\tilde X$ has constant curvature 1 (cf.~also~\cite[(3.3.5)]{Petruninsemi} for a closely related result). But then the action of $G$ on $\mathbb S^n$ is polar.  This contradicts our assumptions.
\end{proof}

\section{Special reducible case}\label{special-red}
In this section we are going to analyze sums of representations of cohomogeneity one.
Herein it is convenient to adopt a different
notation for representations (compare~\cite{str}).
We are going to prove:

\begin{prop} \label{red}
Let $\rho:G\to \mathbf O(V)$ be a representation of a 
compact connected Lie group~$G$. Assume $V=V_1\oplus V_2$ is a 
$G$-irreducible decomposition such that the restrictions
$\rho_i$ to $V_i$ have cohomogeneity $1$.  If $\rho $ is infinitesimally polar but 
not polar then $\rho$ is listed in the table below  (orbit-equivalent representations
are designated by the same number and different letters).  In particular, the cohomogeneity of $\rho$ in $V$  is at most $6$.


\[ \begin{array}{|c|c|c|c|c|}
\hline
\textsl{Case} & G & \rho & \textsl{Cohom} & \textsl{Conditions} \\
\hline
1a& \SO n & \rho_n\oplus\rho_n & 3  & n\geq2\\
1b&\G & \R^7\oplus\R^7 & 3 & - \\
1c&\Spin7 & \Delta_7\oplus\Delta_7 & 3 & - \\
2&\Spin9 & \Delta_9\oplus\Delta_9 & 4 & - \\
3&\SU 2 & \mu_2\oplus\mu_2 & 5  & - \\
4a&\SU n & \mu_n\oplus\mu_n & 4  & n\geq3\\
4b&\U n & \mu_1\otimes\mu_n\oplus\mu_1\otimes\mu_n & 4  & n\geq2\\
5a&\U1\times \SU n\times\U1 & \mu_1\otimes\mu_n\oplus\mu_n\otimes\mu_1 & 3 & n\geq2 \\
5b&\U1\times \SU n & (\mu_1)^r\otimes\mu_n\oplus(\mu_1)^s\otimes\mu_n & 3 & n\geq3,\ r\neq s \\
6&\SP n & \nu_n\oplus\nu_n & 6  & n\geq2\\
7&\SP n\times\U1 & \nu_n\otimes\mu_1\oplus\nu_n\otimes\mu_1 & 5  & n\geq2\\
8&\U1\times \SP n\times\U1 & \mu_1\otimes\nu_n \oplus \nu_n\otimes\mu_1 & 4 & n\geq2 \\
9a&\SP1\times\SP n\times\SP1 & \nu_1\otimes\nu_n\oplus\nu_n\otimes\nu_1& 3  & n\geq2\\
9b&\U1\times\SP n\times\SP1 & \mu_1\otimes\nu_n\oplus\nu_n\otimes\nu_1 & 3  & n\geq2\\
9c&\SP n\times\SP1 & \nu_n\otimes\nu_1\oplus\nu_n & 3  & n\geq2\\
10&\SP n\times\SP1 & \nu_n\otimes\nu_1\oplus\nu_n\otimes\nu_1 & 4  & n\geq2\\

\hline
\end{array} \]
\begin{center}
\textsc{Table 2: Good orbifold quotients}
\end{center}

\[ \begin{array}{|c|c|c|c|c|}
\hline
\textsl{Case} & G & \rho & \textsl{Cohom} & \textsl{Conditions}\\
\hline
11&\U1 & (\mu_1)^r \oplus (\mu_1)^s& 3 & \mbox{$r\neq s$, $rs\neq0$} \\
12&\SU2\times\U1 & \mu_2\otimes\mu_1 \oplus \rho_3 & 3 & - \\
13a&\SP2\times\SP1 & \nu_2\otimes\nu_1\oplus \rho_5 & 3 & -\\
13b&\SP2\times\U1 & \nu_2\otimes\mu_1\oplus \rho_5 & 3 & -\\
13c&\SP2 & \nu_2\oplus \rho_5& 3 & -\\
14&\Spin9 & \Delta_9 \oplus \rho_9 & 3 & -\\
\hline
\end{array} \]
\begin{center}
\textsc{Table 3: Bad orbifold quotients}
\end{center}
\end{prop}

\Pf   The representation $\rho =\rho_1 \oplus \rho _2$, with $\rho_i$ being of cohomogeneity $1$, is polar if and only if it has cohomogeneity $2$.  This happens if and only if $\rho$ is orbit equivalent to the outer direct sum of the representations $\rho_1$ and 
$\rho _2$.  In particular, it is the case if $G$ splits as $G=G_1 \cdot G_2$,
such that the restriction of $\rho_i$ to $G_i$ acts on $V_i$ with cohomogeneity $1$, for~$i=1$,~$2$.

The list of representations of cohomogeneity one
is well known (see e.g.~\cite[\S12.7]{GL}). Cases~1-4, 6, 7 and~10
correspond to a doubling representation. The circle group 
yields case~11 (and~1a). 

Assume now that $G$ is not a circle and $\rho$ not a doubling representation.
If $G$ is simple, then $\rho_1$ and $\rho_2$ must be  two inequivalent representations of 
cohomogeneity one.   We get  cases 13c and 14; the remaining cases are
 representations
of $\SU4$, $\Spin7$ and $\Spin8$ 
of cohomogeneity~$2$ (cf. \cite[p.84]{str2})
and one additional representation $(\SU2,\mu_2 \oplus \rho_3)$.
The latter representation has a non-polar slice
at $p\in\R^3$ given by $(\U1,\mu_1\oplus\mu_1\oplus\theta)$, where $\theta$ denotes a 
trivial summand.

 
Otherwise two factors of $G$ act effectively on one summand, which we assume to be $V_1$. The first case is $(G/\ker\rho_1,V_1)=(\U n,\mu_n)$, $n\geq2$.  Since $\SU n$ acts with cohomogeneity $1$ on $V_1$,  the 
product of the complementary  factors  of $\SU n$ in $G$ does not act with cohomogeneity $1$ on $V_2$ (due to the observation in the beginning of the proof).
We get cases 5a and 12, the representation   $(\U4,\mu_4\oplus\rho_6)$,
which is polar,  and the following candidates: 

\[ \begin{array}{|c|c|c|c|}
\hline
G & \rho & \textsl{Cohom} & \textsl{Conditions}\\
\hline
\U1\times\SU n & (\mu_1)^r\otimes\mu_n \oplus (\mu_1)^s\otimes\mu_n& 4 & \mbox{$n=2$ or $r=s$} \\
\U1\times\SU n & (\mu_1)^r\otimes\mu_n \oplus (\mu_1)^s\otimes\mu_n& 3 & \mbox{$n\geq3$ and $r\neq s$} \\
\hline
\end{array} \]
Since $r=s$ is a doubling representation, for the first representation 
we may assume $n=2$; taking a slice representation at a point $p\in V_1$   shows that $\rho$ can be
infinitesimally polar if and only if $r=\pm s$; as real representations,
the two cases are equivalent one to the other and, after dividing by a finite 
ineffective group, we may assume $r=s=1$.
The second representation yields case~5b.


The second case is $(G/\ker\rho_1,V_1)=(\SP n\times\U1,\nu_n\otimes\mu_1)$,
$n\geq2$.  Again,  the 
product of the complementary  factors  of $\SP n$ in $G$ does not act with cohomogeneity $1$ on $V_2$.
We get cases~8, 9b and~13b plus 
$(\U1 \times\SP n,(\mu_1)^r\otimes\nu_n \oplus (\mu_1)^s\otimes\nu_n)$, 
where $n\geq2$. 
This representation has cohomogeneity~$5$. 
Taking a slice representation shows that it can be
infinitesimally polar if and only if $r=\pm s$; as real representations, 
the two cases are equivalent one to the other.

The last case is $(G/\ker\rho_1,V_1)=(\SP n\times\SP1,\nu_n\otimes\nu_1)$,
$n\geq1$. Again,  the 
product of the complementary  factors  of $\SP n$ in $G$ does not act with cohomogeneity $1$ on $V_2$.
We get $(\SP1\times\SP1,\nu_1\otimes\nu_1\oplus\rho_3)$, which
is polar, plus cases~9a, 9b, 9c and~13a. \EPf

\section{Main step} \label{seclast}
We are going to prove by induction on $k$ the following statement.

\begin{thm} \label{Mainstep}
Let $\rho:G \to \SO {n+1}$ be a non-polar representation of a compact
 Lie group~$G$. 
 Assume that  $X^k=\mathbb S^n /G$ is a good Riemannian orbifold whose universal orbi-covering is diffeomorphic to a sphere.
Then $X^k$ has constant curvature $4$.  Moreover, if $k\geq 3$,  then the restriction of  $\rho$  to the identity component of $G$ 
is the sum   of two representations of cohomogeneity $1$.  Finally, $k\leq 5$.  
%
\end{thm}

\begin{proof}

We proceed by induction on $k$. The case $k=1$ cannot occur, 
since such an action would be polar. The case $k=2$ is well known
(Subsection \ref{subsecstraume}).  Assume now that $k\geq 3$, 
and that the result is true for all $k'<k$.   
By considering the restriction of the action to the identity component of $G$, 
it suffices to prove the result for connected groups.

Either we have the Hopf action of
$\SP1 $ on $\mathbb S^7$, 
for which case all claims are true, or
the quotient has boundary (\lref{posrank} and Subsection \ref{subsecfree}). 
Thus we 
may assume $\partial X\neq \emptyset$.   Due to \cref{induction},  we find an isometric non-polar action of some compact Lie group $N$ on some unit sphere $\mathbb S^m$, such that the 
quotient $Y=\mathbb S^m /N$ is a good Riemannian orbifold of dimension $k-1$, whose universal orbi-covering is diffeomorphic to a sphere.   
By our inductive assumption,
we deduce $k-1 \leq 5$.  Hence $k\leq 6$.  By the inductive assumption, $Y$ has constant curvature $4$, hence its universal orbi-covering  $Z$ is  a sphere of constant curvature $4$.

We first assume that~$\rho$  is reducible. 
Due to \pref{red} and Section \ref{secdiscuss} we only need to prove that $\R^{n+1}$ is the sum of 
two summands on each of which $G$ acts with cohomogeneity one, since this would imply that the quotient has curvature $4$  and dimension $\leq 5$.  Thus it suffices to 
prove  that $G$  acts transitively on any non-trivial invariant great subsphere of $\mathbb S^n$.

Let $\tilde X$ denote the universal orbi-covering of $X$. Then  $\tilde X$ is a sphere 
with curvature $\geq 1$. By \cref{induction},  it admits a reflection $r:\tilde X\to \tilde X$, whose set of fixed points $Z$ is a $(k-1)$-dimensional sphere of constant curvature $4$.  

We claim that 
all geodesics in $\tilde X$ are closed of period $\pi$ 
(cf.~\cite[Section~2.1]{wilking-parity}).
Indeed, the element $-Id \in O (n+1)$ commutes with $G$ and induces an isometry $I$ on 
$X=\mathbb S^n/G$, which sends a point $p$ to the point $\exp (\pi h)$, 
where
$\exp$ is the exponential map of the orbifold $X$ and $h$ is any unit 
vector at $p$.  Hence, the map  $\tilde I:\tilde X\to \tilde X$ which sends a point
$p$ to the endpoint of any geodesic of length $\pi$ starting at  $p$ is  well-defined, and 
it is an isometry of $\tilde X$.   Note that $\tilde I ^2$ is the identity, and  that $\tilde I$ sends any geodesic to itself  without changing the orientation of the geodesic. 
All geodesics in $Z$ are closed of period $\pi$, thus the map $\tilde I$ fixes all points in $Z$.   Hence $\tilde I$ is the identity or the  reflection $r$.
But $r$ changes the orientation  on any geodesic intersecting $Z$ orthogonally.  Hence $\tilde I$ is the identity. Thus all geodesics
in $\tilde X$ are closed of period $\pi$.

In particular, the diameter of $\tilde X$ is at most $\frac {\pi} 2$.  For any point $p\in \tilde X$,
we denote by  $Ant (p)$ the set
of all points $q\in \tilde X$ with $d(p,q) \geq \frac {\pi} 2$.
We claim that  for any point $p \in \tilde X$,  the set  $Ant (p)$ has at most one point.

To prove the claim, note first that
the set  $Ant (p)$ contains with any pair of points $q_1,q_2$ any 
(non necessarily  minimizing) geodesic of length $<\pi$ between these points
by Toponogov's theorem.
Since geodesics  in $\tilde X$ are closed of period $\pi$, we see that $Ant (p)$ contains
any closed geodesic  going through any pair of its different points, hence  $Ant (p)$  is a totally geodesic submanifold of $\tilde X$.  Therefore, if $Ant (p)$ intersects $Z$ in at least $2$ points, it must contain $Z$.  Hence $p\not\in Z$ and  the distance
of $p$ to $r(p)$ is larger than $\frac \pi 2$, since $Z$ separates $\tilde X$. This contradicts the diameter bound.

Thus $Z$ intersects $Ant (p)$ in at  most one point.    Assume that $Ant (p)$ contains at least
two points, hence at least one closed geodesic $\gamma$. 
Now, the distance function to the subset  $Z$ restricted to any of the two open subspaces $\tilde X ^{\pm}$  in which $Z$ divides $\tilde X$ is strictly concave. Hence the closed 
geodesic $\gamma$ cannot be contained in any of this half-spaces of $\tilde X$. Thus,
$\gamma$ intersects $Z$ in some point $q$.
 But   by topological reasons, $\gamma$ cannot intersect $Z$ only at one point $q$.
This contradiction proves the claim  that  $Ant (p)$ consists of at most one point.

Let now  $S= \mathbb S ^l, 0\leq l <n$ be a   
$G$-invariant  great sphere in $\mathbb S^n$ and let $S^{\perp}$ be the orthogonal complement of $S$.   Let $\bar S$ and $\bar {S^{\perp}}$ be the projections of these great spheres to $X$. Then for any point $p\in \bar S$ and any point $q\in \bar {S^{\perp}}$, we have $d(p,q) = \frac \pi 2$.     Since the orbi-covering $\tilde X\to X$ is $1$-Lipschitz,  we deduce that the sets $\bar S$ and $\bar {S^{\perp}}$  have only one element each (otherwise we would get in $\tilde X$ subset of the form $Ant (p)$ or $Ant (q)$ with 
more than one point).  But this is exactly the statement that $G$ acts transitively on $S$ and on 
$S^{\perp}$.

It remains only to prove that~$\rho$ cannot be irreducible.  
Suppose, to the contrary, $\rho$ is 
an irreducible non-polar representation of cohomogeneity 
$4\leq k+1 \leq 7$ on $V=\R ^{n+1}$.
In view of Lemma~\ref{conn-irred-reduction} and the subsequent lines, 
we may replace 
this action by the identity component of 
its minimal reduction and assume that it has trivial principal 
isotropy groups. Due to Lemmas~\ref{iso-conn} and ~\ref{slice}, 
all isotropy groups are connected and products of groups of rank one.
In the cases $k=3$, $4$, we refer 
to~Tables~1 and~2 in~\cite{GL}, in which all
irreducible representations of cohomogeneity $4$ and $5$ are classified, 
and check that each representation 
in the tables is not infinitesimally polar. Indeed, we need
only check representations with \emph{trivial copolarity}  (i.e., those equal to their minimal reductions) and 
non-empty boundary which leaves us with the representation of 
$\SO3\times\U2$ only, which is not 
infinitesimally polar by~\cite[\S8]{GL}.   

Now we may assume $k=5$ or $6$.
In the case $G$ is simple, we refer to the table of non-polar representations
of cohomogeneity~$6$ or $7$
in~\cite[\S12.8]{GL}. We exclude the representations of complex
type by noting that they have toric reductions and thus are not  
infinitesimally polar (cf.~Subsection~\ref{reduction}). 
We exclude $(\SU2,\C^4)$ by noting that 
its orbit space has empty boundary (it produces the smallest weighted quaternionic projective space, namely, the Hitchin orbifold $O_3$), and the remaining 
representations of quaternionic type by noting that they have an isotropy group 
(corresponding to a highest weight vector)
containing a simple factor of rank bigger than one. 
Finally, the 
remaining representation of $\SO3$ is also excluded
because the slice representation at a $0$-weight vector is non-polar. 
Thus we may assume in addition that $G$ is non-simple. 

To finish the proof, let $p\in \bS^n$ be such that $\mathrm{rank}(G_p)\geq
\mathrm{rank}(G)-1$ (cf.~\cite[Lemma~6.1]{wilking-acta}). 
The triviality of principal isotropy groups implies that 
slice representations are effective. 
Due to Lemmas~\ref{slice}
and~\ref{iso-conn},  $G_p$ is a product of groups of rank one and there exists a normal
simple factor $H$ of $G_p$, $H\cong\U1$ or $H\cong\SP1$, corresponding
to a boundary stratum $St$ of $X$. 
Moreover, we can choose $H\cong \SP 1$, unless $G_p$ is a torus. According to Corollary~\ref{induction} 
and its proof, 
there exists a non-polar isometric action of a
compact connected Lie group $(N/H)^0$ on some unit sphere $\bS^m$ such that 
the quotient $Y$ is a good Riemannian orbifold of dimension~$k-1$ and 
its universal orbi-covering is diffeomorphic to a sphere, where
$N$ is the normalizer of $H$ in $G$. By induction, we know that
the action of $(N/H)^0$ on~$W=\R^{m+1}$ is reducible and thus listed
in Table~2 or~3 in~Section~\ref{special-red}.
Using the fact that~$k=5$ or~$6$ we deduce 
that $(N/H)^0$ is $\SU2$ ($k=5$) or $\SP q$ ($k=6$) or 
$\SP q\times\U1$ ($k=5$) with their respective doubling representations.  
Since  the isotropy groups of the action of $(N/H)^0$ on $\bS^m$ 
are still products of groups of rank $1$, in the latter 
two cases we get $q=2$.  

By the choice of~$H$, the group
$G_p$ is contained in $N$. Then $G_p/H$ is the isotropy group 
of $(N/H,W)$ at~$p\in W$. By the explicit form of this
representation, we have $\mathrm{rank}(G_p/H)<\mathrm{rank}(N/H)$
so $\mathrm{rank}(G_p)\leq\mathrm{rank}(N)-1$.
Therefore $\mathrm{rank}(G)=\mathrm{rank}(N)$. Note also that $N\neq G$. 
Next we examine each case separately. 

If  $(N/H)^0=\SU2$ then $\mathrm{rank}(G)=2$.
The representation $(\SU2,\C^2\oplus\C^2)$ has all
isotropy groups trivial, so $G_p=H$.  
Since $G$ is not simple and not equal to~$N$, we deduce $H=\U1$ and 
$G$ is locally isomorphic to~$\SU2\times\SU2$. 
Since the cohomogeneity of $\rho$ is $6$,
we have $\dim V=12$. We are lead to a contradiction by~\cite[\S10.5]{GL}
which says that this representation has empty boundary.

In both the remaining cases,  we see that the  action $(N/H) ^0$ on $W$ 
has a non-Abelian isotropy group.  More precisely, there is a point 
$q\in W$, such that the isotropy group $K=((N/H)^0)_q$ has a non-Abelian
connected component $K^0$ with  
$\mathrm{rank} (K^0) = \mathrm{rank} (N/H) ^0 -1$.
Thus the isotropy group $G_q$ is non-Abelian as well and satisfies 
$\mathrm{rank} (G_q)\geq \mathrm{rank}(N_q)\geq \mathrm{rank} (N) -1 =\mathrm{rank} (G)-1$.   
Thus we may replace $p$ by $q$ 
and, applying \lref{slice}, assume that 
$H$ is isomorphic to  $\SP 1$.


If $(N/H)^0=\SP2$ then $\mathrm{rank}(G)=3$.  
The group $G$ is non-simple, and contains a subgroup locally 
isomorphic to $H\times \SP 2 \cong \SP 1 \times \SP 2$.
We deduce that $G$ is finitely covered 
by  $\tilde G= \SP 1 \times \SP 2$
and $H\cong\SP 1$ is a normal subgroup of $G$, 
 which contradicts the fact that $N\neq G$.


If $(N/H)^0=\SP2\times\U1$ then $\mathrm{rank}(G)=4$.
We see that $G$ contains a subgroup locally isomorphic to 
$H\times \SP2\times\U1$, with 
$H\cong \SP 1$. Moreover, $H$ is not a normal subgroup 
of $G$, since $G\neq N$.

If $G$ has no simple factor of rank~$3$, then it is locally
isomorphic to $\SP2\times G'$, where $G$ and $N$ have a common
$\SP2$-factor; since $W=\Q^2\oplus\Q^2$ is a subspace 
of $V$ invariant under this factor, we deduce that
$V=\Q^2\otimes_{\mathbb F} V'$
where $V'$ is a representation of $G'$ and
$\mathbb F=\R$, $\C$ or $\Q$ according to whether
$V'$ is of real, complex or quaternionic type. 
In any case, 
$\dim_{\mathbb R}V$ is divisible by~$8$. It follows that 
$\dim G'=\dim V'-\dim\SP2-6$ is also divisible by~$8$.
Since $G'$ has rank~$2$, we deduce that $G'=\SU3$. 
Now $\dim V=\dim\SP2\times\SU3+6=24$, so 
$(G,V)=(\SP2\times\SU3,\C^4\otimes_{\mathbb C}\C^3)$  
and $H$ must be an $\SU 2$-subgroup of $\SU3$.  
The fixed point set $Z$ of any such subgroup on $\C^3$ has 
complex dimension  $1$, so the set of fixed points
of $H$ in $\C^4 \otimes_{\mathbb C} \C^3$ is $\C^4\otimes_{\mathbb C}Z$. 
Hence its complex dimension is not~$8$, a contradiction.

Suppose next $G$ has a simple factor $G'$ of rank $3$. 
Then $G'$ contains $\SP2$ and an $\SP1$-subgroup  normalizing it. 
It follows that  $G'=\SP3$. 


Now $G=\SP3\times\SP1$ and $\dim V=\dim G+6=30$
or $G=\SP3\times\U1$ and $\dim V=\dim G+6=28$.
Since $30$ is not divisible by $4$, in the first case
$V$ is a real tensor product of irreducible representations 
of real type of $\SP3$ and $\SP1$. 
The irreducible representations of $\SP3$ of lowest degree 
(i.e.~complex dimension) are the fundamental representations
(cf.~\cite[Lemma~3.1]{On})
with degrees $6$, $14$ and $14$, and are
respectively of quaternionic, real, quaternionic type. Since $30$ is not divisible by $14$, 
$G=\SP3\times\SP1$ is ruled out. 
On the other hand, in the second case we get two candidates
$(\SP3\times\U1,\Lambda^2\C^6\ominus\C)$ and
$(\SP3\times\U1,\Lambda^3\C^6\ominus\C^6)$. 
The former representation is excluded because the isotropy group
at a point corresponding to (the projection along the trivial 
summand $\C$ of) a skew-symmetric bilinear form 
of rank~$2$ contains a rank $2$ subgroup isomorphic to $\SP2$,
and the latter one is excluded because the isotropy group 
at a highest weight vector contains a rank~$2$ subgroup 
isomorphic to $\SU3$; both situations contradict Lemma~\ref{slice}.
\end{proof}

\section{Conclusion} \label{secverylast}

Now it is easy to deduce all the results stated in the introduction.  
Note that Theorem~\ref{good} has already been proved in Section~\ref{secorbifold}. 

Under the assumptions of \tref{thirdthm}, we see that the quotient $X$
has non-empty boundary (\lref{posrank} and Subsection \ref{subsecfree}). By \tref{good}, $X$  is a good orbifold, whose universal orbi-covering is diffeomorphic to the sphere.  Therefore, we may apply \tref{Mainstep} to see that $\rho$ is the sum of two representations of cohomogeneity one.   Now, \tref{thirdthm} follows from 
\pref{red} and the description of the quotients in Section \ref{secdiscuss}.

Under the assumption of \tref{thmsum} we  may assume that $G=G^0$.   In all cases listed in the theorem, the quotient $X=\mathbb S^n /G$ is a Riemannian orbifold whose geometry is discussed in Section \ref{secdiscuss}.  On the other hand,  assume that $X$  is a Riemannian orbifold of dimension at least $3$ (so that
$\rho$ has cohomogeneity at least $4$) and that $\rho$ is not polar.
If the quotient $X$ has no boundary then $\rho$  is almost free  
(\lref{posrank}) hence  as stated (Subsection \ref{subsecfree}). 
If $X$ has boundary  then  $\rho$ satisfies the assumption 
of \tref{Mainstep} (due to \tref{good}).  Hence $\rho$ is  one of 
the representations appearing  in \pref{red};
in the only case in Table~2 of rank one, the quotient has no boundary, 
so this case cannot occur. It follows that $G$ has rank
at least $2$ and $\rho$ is as described in \tref{thirdthm}.
This proves \tref{thmsum}.

 \tref{secondthm} is a direct consequence of \tref{thmsum}. Finally, to prove
\tref{firstthm} we may again assume that $G$ is connected, since $\mathbb S^n /G$ and
$\mathbb S^n /G^0$ have the same universal
orbi-covering. Thus we may apply \tref{thmsum}.
If $\rho$ is polar then $X$ has constant curvature $1$ and the universal orbi-covering $\tilde X$ of $X$ is the sphere of constant curvature $1$.  
If  $G$ has rank one and $X$ has empty boundary
then $X$ is a weighted projective space as discussed in Subsection \ref{subsecfree}.  Moreover, $X$ is simply connected as an orbifold in this case, 
hence $\tilde X=X$.  If $X$ has dimension $2$ and curvature $\neq 1$, then as proved by Straume~\cite{str}, the quotient 
$X$ can be represented as the quotient $X= \mathbb S^3 /G'$, for some group $G'$ of 
dimension $1$.  Thus $X$ has as an orbi-covering $X'=\mathbb S^3 /\U 1$,  where $\U 1$
acts without fixed points. In such a case $X'$ is  a weighted complex projective space, which is then the universal orbi-covering of $X$.
Finally, if $X$ has constant curvature $4$ as in Theorem~\ref{thirdthm}, then
$\tilde X$ is a sphere of constant curvature $4$.

\providecommand{\bysame}{\leavevmode\hbox to3em{\hrulefill}\thinspace}
\providecommand{\MR}{\relax\ifhmode\unskip\space\fi MR }
\providecommand{\MRhref}[2]{%
  \href{http://www.ams.org/mathscinet-getitem?mr=#1}{#2}
}
\providecommand{\href}[2]{#2}


\end{document}